\documentclass[10pt]{article}
\usepackage{amsmath, amsfonts, amsthm, amssymb, amscd}
\usepackage[mathcal]{euscript}
\usepackage{dsfont, pxfonts, mathrsfs, accents, verbatim}
\theoremstyle{plain}
        \newtheorem{theorem}{Theorem}[section]

        \newtheorem{lemma}[theorem]{Lemma}
        \newtheorem{corollary}[theorem]{Corollary}
        \newtheorem{remark}[theorem]{Remark}

\numberwithin{equation}{section}

\newcommand \be     {\begin{equation}}
\newcommand \ee     {\end{equation}}

\newcommand \del        \partial
\newcommand \eps        \varepsilon
\newcommand \auth   \textsc

\begin{document}

\title{A Conformally Invariant Classification Theorem in Four Dimensions}
\author{Bing-Long Chen,  Xi-Ping Zhu\footnote{ Department of Mathematics, Sun Yat-Sen University,
Guangzhou, P. R. of China. E-mails: {\sl mcscbl@mail.sysu.edu.cn., stszxp@mail.sysu.edu.cn., }
\newline
2000\textit{\ AMS Subject Classification:} 53C25, 53C44.
 }
}

\date{}

\maketitle

\begin{abstract}
In this paper, we prove a classification theorem of 4-manifolds according to some conformal invariants, which generalizes the conformally invariant sphere theorem of Chang-Gursky-Yang \cite{CGY}. Moreover, it provides a four-dimensional analogue of the well-known classification theorem of Schoen-Yau \cite{SY2} on 3-manifolds with positive Yamabe invariants.
\end{abstract}


\section{Introduction}

\label{IN-0} One of the main themes in geometry is to classify the topology of the manifolds in terms of certain topological or differential invariants. In dimension 2, the invariant can be taken to be the Euler characteristic. The classical uniformization theorem implies that if the Euler characteristic of a closed surface is positive, then the surface must be diffeomorphic to the sphere $S^2$ or the real projective space $RP^2.$   Note that the Gauss-Bonnet theorem expresses the Euler characteristic number $\chi(M^2)$  of a closed surface $M^2$ as an integral of the Gaussian curvature, i.e., $$\frac{1}{2\pi}\int_{M^2} K d\sigma=\chi(M^2).$$  In dimension $\geq 3$, one can consider an analogous invariant, called Yamabe invariant, as follows.

Fix a
 differentiable manifold $M^n$ of real dimension $n$ without boundary. Given a metric $g$ on $M^n,$  let $\mathcal{C}_g=\{\rho g\mid \rho>0\}$ be the class of metrics conformal
 to a fixed metric $g.$  Define
 \be \label{yyb'} \mathcal{Y}(M^n,\mathcal{C}_g)=\inf\limits_{g'\in
\mathcal{C}_g}\frac{\int_{M^n}R_{g'}dv_{g'}}{(\int_{M^n}dv_{g'})^{\frac{n-2}{n}}}
\ee
 and the Yamabe invariant $\mathcal{Y}(M^n)$ of the manifold is defined to be $$
\mathcal{Y}(M^n)=\sup_{\mathcal{C}}\mathcal{Y}(M^n,\mathcal{C}),
$$
 where the superum is taken over all conformal classes of Riemannian metrics.
Particularly in dimension 2, $\mathcal{Y}(M^2)=2\pi \chi(M^2)$ by Gauss-Bonnet theorem.

 The following well-known theorem due to Schoen-Yau \cite{SY2} gives a complete classification of compact three-dimensional manifolds with positive Yamabe invariant $\mathcal{Y}(M^3).$

 \begin{theorem}(Schoen-Yau \cite{SY2}, Perelman \cite{P2})\label{tschyp}

Let $M^3$ be a compact three-dimensional manifold with $\mathcal{Y}(M^3)>0.$
Then $M^3$ is diffeomorphic to a $S^3/ \Gamma $ or $(R\times S^2)/\Gamma'$ or a connected sum of several of these manifolds,
where $\Gamma$ is a discrete subgroup of the isometries of $S^3,$ $\Gamma'$ is cocompact discrete subgroup of the isometric group
of $R\times S^2.$
\end{theorem}

 When dimension $n\geq 4,$  the Yamabe invariant alone is too weak to control the topology of the manifolds. One needs additional assumptions to investigate the topology of the manifolds with positive Yamabe invariant $\mathcal{Y}(M)$.

In 2003, S. Y. A. Chang, M. Gursky and P. Yang \cite{CGY} proved a conformally invariant sphere theorem in dimension 4.
In their theorem, besides the positivity of the Yamabe invariant $\mathcal{Y}(M^4),$ they assumed the Weyl curvature is suitably controlled in $L^2$ sense by the Euler characteristic $\chi(M^4)$ of the manifold.

Before the precise statement of their result, we recall some basic materials in dimension 4. A local orientation of a Riemannian 4-manifold gives  a splitting
$\Lambda^{2}=\Lambda^2_{+}\oplus \Lambda^2_{-}$ of 2-forms into
self-dual and anti-self dual 2-forms. Thus the curvature operator has a block
decomposition
$$
Rm=\left(\begin{array}{cc}
A&B\\ {}^tB&C\end{array}\right)$$
where $A=\frac{R}{12}+W_{+},$  $C=\frac{R}{12}+W_{-},$ $B=\overset{\circ}{Ric}.$ Here $W_{+}$ and $W_{-}$ are the self-dual and anti-self-dual Weyl curvature tensors respectively, while $\overset{\circ}{Ric}$ is the trace free part of the Ricci curvature tensor.
The Gauss-Bonnet-Chern theorem says
\begin{equation}\label{GB}
\begin{split}
\frac{1}{8\pi^2}\int_{M^4}|W_{+}|^2+ |W_{-}|^2+\frac{R^2}{24}-\frac{|\overset{\circ}{Ric}|^2}{2}dvol_g= \chi(M^4).
\end{split}
\end{equation}

We can now state the conformally invariant sphere theorem of Chang-Gursky-Yang:

\begin{theorem}(Chang-Gursky-Yang \cite{CGY})\label{CGY}

Let $(M^4,g)$ be a compact four-dimensional Riemannian manifold. Suppose we have
\begin{equation}\label{cip}
\begin{split}
& \ (i)\ \ \ \ \mathcal{Y}(M^4,\mathcal{C}_g)>0,\\
& \ (ii)\int_{M^4}|W_{+}|^2+ |W_{-}|^2dvol_g<4\pi^2 \chi(M^4).
\end{split}
\end{equation}
Then $M^4$ is diffeomorphic to a  $S^4$ or $RP^4.$
\end{theorem}

Note that the condition (ii) (and (i)) is invariant under conformal change of the metric. We emphasize that the norms $|W_{+}|,$  $|W_{-}|$ used here are one half  of the norms for Weyl tensors used in the paper \cite{CGY}.  Moreover, in the same paper \cite{CGY}, they also obtained the following rigidity theorem that shows the pinching condition (ii) is sharp.
\begin{theorem}(Chang-Gursky-Yang \cite{CGY})\label{CGYR}

Let $(M^4,g)$ be a compact four-dimensional Riemannian manifold which is not diffeomorphic to $S^4$ or $RP^4.$ Suppose we have
\begin{equation}\label{cip}
\begin{split}
& \ (i)\ \ \ \ \mathcal{Y}(M^4,\mathcal{C}_g)>0,\\
& \ (ii)\int_{M^4}|W_{+}|^2+ |W_{-}|^2dvol_g=4\pi^2 \chi(M^4).
\end{split}
\end{equation}
Then one of the following must be true:\\
1)  $(M^4, g)$ is conformal to $CP^2$ with the Fubini-Study metric , or\\
2)  $(M^4, g)$  is conformal to a manifold which is isometrically covered by $S^3\times S^1$ endowed with
the standard product metric.
\end{theorem}


It is obvious that the Euler characteristic $\chi(M^4)$ in the Theorem \ref{CGY}  has to be positive; and in the rigidity Theorem \ref{CGYR}, $\chi(M^4)$ is non-negative.  One of the main goal of this paper is to generalize the above sharp conformally invariant sphere theorem  to manifolds with possibly non-positive Euler characteristic. We state our first result in the following theorem:

\begin{theorem}\label{t1.1'}
Let $(M^4,g)$ be a compact four-dimensional Riemannian manifold.  Suppose we have
\begin{equation}\label{cip}
\begin{split}
& \ (i)\ \ \ \ \mathcal{Y}(M^4,\mathcal{C}_g)>0,\\
& \ (ii)\int_{M^4}[\max\{\lambda_{\max}(W_{+}),\lambda_{\max}(W_{-})\}]^2 dvol_g<\frac{1}{36}\mathcal{Y}(M^4,\mathcal{C}_g)^2,
\end{split}
\end{equation}
where  $\lambda_{\max}(W_{+})$ (or $\lambda_{\max}(W_{-})$)
is the largest eigenvalue of the Weyl operator $W_{+}$ (or $W_{-}$) acting on the
(anti-)self-dual 2-forms.
Then $M^4$ is diffeomorphic to a connected sum  $$S^4\# m RP^4\#
(R\times S^3)/\Gamma_1\#\cdots (R\times S^3)/\Gamma_k,$$
for  $m=0,$ or $1$ and some nonnegative integer $k$, where each $\Gamma_i$ is a cocompact discrete subgroup of the isometric group
of $R\times S^3.$
\end{theorem}

Clearly, the condition (ii) in Theorem \ref{t1.1'} (and (i)) is of conformally invariant. We now show that Theorem \ref{CGY} can be deduced from Theorem \ref{t1.1'}.

  Indeed by Gauss-Bonnet-Chern theorem, the condition (ii) in Theorem \ref{CGY} is equivalent to
  \begin{equation}\label{eqc}
\begin{split}
\int_{M^4}|W_{+}|^2+ |W_{-}|^2 dvol_g<\int_{M^4}\frac{R^2}{24}-\frac{|\overset{\circ}{Ric}|^2}{2}dvol_g.
\end{split}
\end{equation}
  The fact that $$\int_{M^4}|W_{+}|^2+ |W_{-}|^2 dvol_g$$ is conformally invariant implies  $$\int_{M^4}\frac{R^2}{24}-\frac{|\overset{\circ}{Ric}|^2}{2}dvol_g$$ is also conformally invariant by (\ref{GB}).
  By the solution of Yamabe problem due to Schoen \cite{Sch}, there is a Yamabe metric $\tilde{g}$ in the conformal class $\mathcal{C}_g$ of $g$ such that
\be \label{yyb'} \mathcal{Y}(M^4,\mathcal{C}_g)=\inf\limits_{g'\in
\mathcal{C}_g}\frac{\int_{M^4}R_{g'}dv_{g'}}{(\int_{M^4}dv_{g'})^{\frac{1}{2}}}=
\frac{\int_{M^4}R_{\tilde{g}}dv_{\tilde{g}}}{(\int_{M^4}dv_{\tilde{g}})^{\frac{1}{2}}}.
\ee
Moreover, the Yamabe metric $\tilde{g}$ has constant scalar curvature. By the conformal invariance of right hand side of (\ref{eqc}), we have
\begin{equation}\label{eqcc}
\begin{split}
\int_{M^4}|W_{+}|^2+ |W_{-}|^2 dvol_g<& \int_{M^4}\frac{{R_{\tilde{g}}}^2}{24}-\frac{|\overset{\circ}{Ric_{\tilde{g}}}|_{\tilde{g}}^2}{2}dvol_{\tilde{g}}\\
\leq & \int_{M^4}\frac{{R_{\tilde{g}}}^2}{24}=\frac{\mathcal{Y}(M^4,\mathcal{C}_{g})^2}{24}.
\end{split}
\end{equation}
On the other hand, let $\lambda_1\geq\lambda_2\geq \lambda_3$ be the eigenvalues of $W_{+}$. Since $W_{+}$ is of trace free, we have $\lambda_1+\lambda_2+\lambda_3=0$, and
\begin{equation}\label{esr}
\begin{split}
|W_{+}|^2=\lambda_1^2+\lambda_2^2+\lambda_3^2\geq \lambda_1^2+\frac{1}{2}(\lambda_2+ \lambda_3)^2=\frac{3}{2}|\lambda_{\max}(W_{+})|^2.
\end{split}
\end{equation}
Similarly, we have $$ |W_{-}|^2\geq \frac{3}{2}|\lambda_{\max}(W_{-})|^2.$$ Combining with (\ref{eqcc}), we have
\begin{equation}\label{eqccc}
\begin{split}
\int_{M^4}|\lambda_{\max}(W_{+})|^2+ |\lambda_{\max}(W_{-})|^2 dvol_g<& \frac{\mathcal{Y}(M^4,\mathcal{C}_{g})^2}{36}.
\end{split}
\end{equation}
which clearly implies the condition (ii) in Theorem \ref{t1.1'}. Thus it follows from the conclusion of Theorem \ref{t1.1'} that $b_2(M^4)=0$  and $M^4$ has a finite cover $\tilde{M}^4$ diffeomorphic to $\#^{k} S^3\times S^1,$ where $k=b_1(\tilde{M}^4).$
 Obviously, the condition (ii) in Theorem \ref{CGY} implies that $\chi(M^4)>0$. Since $\chi(\tilde{M}^4)=2-2b_1=2-2k>0$, we then have $k=0$. By Theorem \ref{t1.1'} again, we deduce that $M^4$ is diffeomorphic to $S^4$ or $RP^4.$ Hence we have proved that Theorem \ref{CGY} can be deduced from Theorem \ref{t1.1'}.

  Furthermore, it follows from Theorem \ref{t1.1'} that if $\chi(M^4)=0,$ then $M^4$ has a finite cover diffeomorphic to $ S^3\times S^1.$ In the other most cases, we have $\chi({M}^4)<0$ since $\chi(\tilde{M}^4)=2-2b_1=2-2k<0$, for $k>1.$ Thus we have the following remark.

 \begin{remark}

   Theorem \ref{t1.1'} generalizes Theorem \ref{CGY} in two aspects: the conditions in Theorem \ref{CGY} imply the conditions in Theorem \ref{t1.1'}; Theorem \ref{t1.1'} admits  many 4-manifolds with negative Euler characteristic.

 \end{remark}

By combining a result of Micallef-Wang (Theorem 4.10 in \cite{MW}), we also have a rigidity theorem:

\begin{theorem}\label{t1.1''}
Let $(M^4,g)$ be a compact four-dimensional Riemannian manifold.  Suppose we have
\begin{equation}\label{cip}
\begin{split}
& \ (i)\ \ \ \ \mathcal{Y}(M^4,\mathcal{C}_g)>0,\\
& \ (ii)\int_{M^4}[\max\{\lambda_{\max}(W_{+}),\lambda_{\max}(W_{-})\}]^2 dvol_g=\frac{1}{36}\mathcal{Y}(M^4,\mathcal{C}_g)^2.
\end{split}
\end{equation}
If $M^4$ is not diffeomorphic to $$S^4\# m RP^4\#
(R\times S^3)/\Gamma_1\#\cdots (R\times S^3)/\Gamma_k,$$ for  all $m=0, 1$ and nonnegative integer $k,$  then one of the following occurs \\
(a) $(M^4,g)$ is conformal to $CP^2$ with the Fubini-Study metric;\\
  (b) the universal cover of  $(M^4,g)$ is conformal  to $(\Sigma_1,g_1)\times (\Sigma_2,g_2),$ where the surface $(\Sigma_i,g_i)$ has constant Gaussian curvature $k_i,$ and $k_1+k_2>0.$
\end{theorem}

  To put our result into a better formulation, we shall consider a generalized Yamabe invariant $\mathcal{GY}$ on $4-$manifolds, and give a complete classification of all 4-manifolds with positive $\mathcal{GY}.$ More precisely, analogous to the Yamabe invariant on a
 differentiable manifold $M^n$, a conformal invariant $\mathcal{GY}(M^n,\mathcal{C})$ and a differentiable invariant
$\mathcal{GY}(M^n)$ can be defined by: 
\be \begin{split}\label{cfi}
&\mathcal{GY}(M^n,\mathcal{C})=\inf_{g\in \mathcal{C}}\frac{\int_{M^n}(R_g-6\max\{\lambda_{\max}(W_{+}),\lambda_{\max}(W_{-})\}) dv_{{g}}}{{(\int_{M^n}dv_{{g}})}^{\frac{1}{2}}}, \\
  & \mathcal{GY}(M^n)=\sup_{ \mathcal{C}}\mathcal{GY}(M^n, \mathcal{C}).
   \end{split}
\ee

\vskip 0.5cm

\begin{theorem}\label{t1.1}
Let $M^4$ be a compact four-dimensional manifold with $\mathcal{GY}(M^4)>0.$
Then $M^4$ is diffeomorphic to a connected sum  $$S^4\# m RP^4\#
(R\times S^3)/\Gamma_1\#\cdots (R\times S^3)/\Gamma_k,$$
for  $m=0,$ or $1$ and some nonnegative integer $k,$ where each $\Gamma_i$ is a cocompact discrete subgroup of the isometric group
of $R\times S^3.$
\end{theorem}

With Theorem \ref{t1.1} in hand, Theorem \ref{t1.1'} can actually be proved by verifying $\mathcal{GY}>0.$ Note that the Weyl tensor vanishes on three-dimensional manifolds. Clearly, Theorem \ref{t1.1}  is a four-dimensional analogue of Schoen-Yau's classification Theorem \ref{tschyp}.

The proof of Theorems \ref{t1.1'}, \ref{t1.1''} and  \ref{t1.1}  will be given in the section 3.  Let us first recall the strategy of the proof of Chang-Gursky-Yang's conformally invariant sphere theorem in \cite{CGY}.  In the paper \cite{Mar}, Margerin showed that the Ricci flow will deform a metric (on compact four-manifolds) satisfying the condition
   \begin{equation}\label{ppcn}
\begin{split}
|W_{+}|^2+ |W_{-}|^2 <\frac{R^2}{24}-\frac{|\overset{\circ}{Ric}|^2}{2}
\end{split}
\end{equation}
   to a constant curvature one.  The proof of Chang-Gursky-Yang's Theorem \ref{CGY} is to find a conformal factor to transform  the integral  condition (\ref{eqc}) into the above pointwise curvature pinching condition (\ref{ppcn}). This amounts to solving a fully nonlinear equation of Monge-Ampere type. It relies on their previous works on \cite{CGY1} and \cite{CGY2}. The proof of our Theorem \ref{t1.1} also consists of two steps. The first step is to transform the integral condition $\mathcal{GY}>0$ to a pointwise curvature condition by solving a semi-linear elliptic equation. This step is much easier than that of Theorem \ref{CGY}. Fortunately, the pointwise curvature condition in our proof is exactly the positive isotropic curvature (PIC) condition which we have studied in \cite{CZ} and \cite{CTZ}.  The second step depends on our (joint with S.H.Tang ) classification theorem \cite{CTZ} on four-manifolds with PIC. This classification for special four-manifolds without essential incompressible space forms was initiated by Hamilton in \cite{H} and completed in \cite{CZ}.

 \begin{remark} In \cite{Gur}, the corresponding functional and conformal invariant with the modified scalar curvature $R_g-6\max\{\lambda_{\max}(W_{+}),\lambda_{\max}(W_{-})\}$ replaced by $R_g-2\sqrt{6}|W_{\pm}|$ in (\ref{cfi}) were introduced by M. Gursky earlier.  Moreover, many variants of scalar curvature based on the concept of conformal weight were also introduced in \cite{GLe}. Associated to every modified scalar curvature, one can define a modified Yamabe invariant similarly. A more important question is to find out the geometrical and topological significance of these modified Yamabe invariants.
\end{remark}

 In the end of this section, we would like to mention an interesting result in symplectic geometry that has the same spirit as ours. It was shown in \cite{Liu} and \cite{OO} that if a compact 4-manifold with positive Yamabe invariant admits a symplectic structure, then the manifold is diffeomorphic to some blow ups of a complex rational surface or a complex ruled surface.

{\bf Acknowledgements} The authors are grateful to S.-H. Tang for many helpful discussions. The authors thank professor M. Gursky brings the literatures \cite{Gur} \cite{GLe} into our attention.  The first author is partially supported by NSFC11025107, the second author by NSFC10831008.

\section{Generalized Yamabe Invariant}

\label{IN-0}

Before introducing the generalized Yamabe invariant on any 4-manifold, we start with some preliminaries on the original Yamabe invariant.

Yamabe invariant arises from a variational problem in
 seeking Einstein metrics on a given manifold. We describe it as follows.  Fix a
 differentiable manifold $M^n$ of real dimension $n$ without boundary. Given a metric $g$ on $M^n,$ we consider
 the quantity
  $$\mathcal{F}(M^n,g)=
 \frac{\int_{M^n}R_{g}dv_{g}}{(\int_{M^n}dv_{g})^{\frac{n-2}{n}}},
 $$
 where $R_g$ is the scalar curvature of $g,$ $dv_{g}$ is the volume
 measure of $g.$ This gives a functional defined on the  space
 of all Riemannian metrics of $M^n.$ When $n\geq 3,$ Einstein metrics correspond to the critical points
 of $\mathcal{F}$. Unfortunately the functional $\mathcal{F}$ is neither upper nor lower bounded, which causes serious difficulty to get a critical point for the functional. It was observed by Yamabe \cite{Y} that the functional $\mathcal{F}$
 has a lower bound when restricted to the class $\mathcal{C}_g$ of metrics conformal
 to a fixed metric $g$. Moreover, the infimum
\be \label{yyb} \mathcal{Y}(M^n,\mathcal{C}_g)=\inf\limits_{g'\in
\mathcal{C}_g}\frac{\int_{M^n}R_{g'}dv_{g'}}{(\int_{M^n}dv_{g'})^{\frac{n-2}{n}}}
\ee
 has an upper bound $\mathcal{Y}(M^n,\mathcal{C}_g)\leq \mathcal{F}(S^n,g_{round})$ by
 a result of Aubin \cite{A}.

The Yamabe invariant is defined as follows
$$
\mathcal{Y}(M^n)=\sup_{\mathcal{C}}\mathcal{Y}(M^n,\mathcal{C})
$$
where the sup is taken over all conformal classes $\mathcal{C}$ of
Riemannian metrics. Clearly $\mathcal{Y}(M^n)\leq \mathcal{Y}(S^n).$ The point is that if $\mathcal{Y}(M^n)$
is achieved by some $\tilde{g}\in \mathcal{C}$ in some conformal class $\mathcal{C}$,  this metric $\tilde{g}$ is
necessarily an Einstein metric.

By the solution of Yamabe problem (see \cite{A} \cite{Sch}), $\mathcal{Y}(M^n,\mathcal{C})$
can always be
 achieved by some metric of constant scalar curvature (so called Yamabe metric) in $\mathcal{C}$.



\vskip 0.5cm
 Instead of the Yamabe invariant, one can introduce some generalized  and stronger Yamabe invariants in a natural manner. The original Yamabe invariant provides a framework on constructing these invariants.  All these constructions are based on some variants of scalar curvature function. In \cite{GLe}, the modified scalar curvature of the form $R_g-f$ was introduced,  where $f$ is some function of conformal weight $-2$.  In this paper, we are only interested in those particular $f$ which are functionals of Weyl tensors. The construction is the following.  

 Assume the dimension $n>2$.  Let $f$ be a fixed nonnegative, invariant function of homogeneity one defined on the space of self-adjoint linear operators  $W: \Lambda^2(R^n)\rightarrow \Lambda^2(R^n)$, i.e. $f(cW)=cf(W)$, for any $c>0$; $f(UWU^{-1})=f(W)$ for any orthogonal matrix $U$.  The Weyl tensor, regarded as an operator $W: \Lambda^2\rightarrow \Lambda^2$, may be considered as an operator on $\Lambda^2(R^n)$ after choosing a frame. So $f(W)$ is well defined since $f$ is invariant under conjugation of $W$ by an orthogonal matrix.

 Now we consider a generalized "scalar curvature" $R_g-f(W_g)$ and a functional
 \be
 \mathcal{F}_{f}(M^n,g)= \frac{\int_{M^n}R_g-f(W_g) dv_{g}}{(\int_{M^n}dv_g)^{1-\frac{2}{n}}}
  \ee
  over the space of Riemannian metrics.

   By setting $\hat{g}=u^{\frac{4}{n-2}}g,$  it is not hard to see
\be\label{ep}
R_{\hat{g}}-f(W_{\hat{g}})=u^{-\frac{n+2}{n-2}}[-4\frac{n-1}{n-2}\triangle u +(R_g-f(W_g)) u],
\ee
and hence

\be \label{2.8}
   \frac{\int_{M^n}R_{\hat{g}}-f(W_{\hat{g}}) dv_{\hat{g}}}{{(\int_{M^n}dv_{\hat{g}})}^{1-\frac{2}{n}}}= \frac{\int_{M^n}(R_g-f(W_g))u^2+4\frac{n-1}{n-2}|\nabla u|^2dv_g}{(\int_{M^n}u^{\frac{2n}{n-2}}dv_g)^{\frac{n-2}{n}}}.\ee
 By Sobolev imbedding theorem, (\ref{2.8}) implies
  $
 \mathcal{F}_{f}(M^n,\cdot)
 $
  has a lower bound on any fixed conformal class $\mathcal{C}.$  Then we may introduce a conformal and a differentiable invariant as well:
  \be \begin{split}
&\mathcal{Y}_{f}(M^n,\mathcal{C})=\inf_{g\in \mathcal{C}}\mathcal{F}_{f}(M^n,g)\\
  & \mathcal{Y}_{f}(M^n)=\sup_{ \mathcal{C}}\mathcal{Y}_{f}(M^n, \mathcal{C}).
   \end{split}
  \ee 

Similar to the Yamabe problem, the minimizing question for the functional $\mathcal{F}_{f}(M^n,g)$ in a fixed conformal class $ \mathcal{C}$ is called a generalized Yamabe problem.
Since the solution of generalized Yamabe problem may have limited regularity, we have to enlarge our conformal class a little. For fixed smooth Riemannian metric $g,$ denote the new conformal class \begin{equation}\hat{\mathcal{C}}_g=\{u^{\frac{4}{n-2}}g\mid u>0, u\in C^{2,\alpha}, \text{for\  any }0<\alpha<1 \}. \end{equation}
Clearly $\hat{\mathcal{C}}_g\supset {\mathcal{C}}_g$ and $\mathcal{Y}_{f}(M^n,\mathcal{C}_g)=\mathcal{Y}_{f}(M^n,\hat{\mathcal{C}}_g)$.

For convenience, we might drop the subscript $g$ from ${\mathcal{C}}_g$ and $\hat{\mathcal{C}}_g$ when there is no danger of confusion.

Simple properties of these invariants are listed in the following lemmas.
The first lemma asserts that the generalized Yamabe problem
is solvable.
\begin{lemma}\label{l1.1}

For any conformal class $\hat{\mathcal{C}}_g$ of Riemannian metrics on
$M^n$, the infimum $\mathcal{Y}_f(M^n,\hat{\mathcal{C}}_g)$ of the functional
$\mathcal{F}_f(M^n,g)$ over $\hat{\mathcal{C}}_g$ can be achieved by some  $\hat g\in \hat{\mathcal{C}}_g$ which has
constant generalized scalar curvature $R_{\hat g}-f(W_{\hat g}).$
\end{lemma}
\begin{proof} First of all, we know $\mathcal{Y}_f(M^n,\mathcal{C})\leq \mathcal{Y}(M^n,\mathcal{C})\leq \mathcal{Y}(S^n).$ The last inequality $\mathcal{Y}(M^n,\mathcal{C})\leq \mathcal{Y}(S^n)$ is a result of Aubin \cite{A}. The proof consists of two steps.

Step 1.  We shall show that the problem is solvable provided
$\mathcal{Y}_f(M^n,\mathcal{C})<\mathcal{Y}(S^n).$  We only give a sketch of proof, the argument is similar to that of the Yamabe problem (see the book \cite{SchoenYau} for detailed exposition).

For $2<s<\frac{2n}{n-2},$ we consider the functional
\be \label{2.8'}
   \mathcal{F}_s(u)=\frac{\int_{M^n}(R_g-f(W_g))u^2+4\frac{n-1}{n-2}|\nabla u|^2dv_g}{(\int_{M^n}|u|^{s}dv_g)^{\frac{2}{s}}},\ee
and $\mu_s=\inf\{\mathcal{F}_s(u)\mid u\in W^{1,2}(M^n)\backslash \{0\} \}.$ By Sobolev embedding theorem, there is a constant $C>0,$ such that $-C<\mu_s<C,$ for any $s\in (2,\frac{2n}{n-2}).$

We know the function $\mu_s$ is upper semi-continuous, in particular, we have $\overline{\lim\limits_{s\rightarrow \frac{2n}{n-2}}}\mu_s\leq \mathcal{Y}_f(M^n,\mathcal{C})$. 
Indeed, take a minimizing sequence $u_i$ of $\mathcal{F}_{\frac{2n}{n-2}}$ such that $\lim\limits_{i\rightarrow \infty}\mathcal{F}_{\frac{2n}{n-2}}(u_i)=\mathcal{Y}(M^n,\mathcal{C}),$  for fixed $i$ and $s$ we have
$$
\mu_s\leq \mathcal{F}_{s}(u_i)\leq \mathcal{F}_{\frac{2n}{n-2}}(u_i)\frac{\parallel u_i\parallel_{L^{\frac{2n}{n-2}}}^2}{\parallel u_i\parallel_{L^{s}}^2}
$$
This implies $\overline{\lim\limits_{s\rightarrow \frac{2n}{n-2}}}\mu_s\leq \mathcal{F}_{\frac{2n}{n-2}}(u_i)$ for any $i.$ The result follows from taking the limit $i\rightarrow \infty.$

For $s\in (2,\frac{2n}{n-2}),$ take a minimizing sequence $u_i\in W^{1,2}(M^n)$  such that $$\int_{M^n} |u|_i^sdv_d=1, \ \ \ u_i\geq 0$$ and $$ \lim\limits_{i\rightarrow \infty}\mathcal{F}_s(u_i)=\mu_s.$$ By Sobolev embedding theorem, there is a subsequence converging to some $u_s$ which satisfies
$$
[-4\frac{n-1}{n-2}\triangle u_s +(R_g-f(W_g)) u_s]=\mu_s u_s^{s-1}
$$
in weak sense. Since $s<\frac{2n}{n-2},$ by Sobolev embedding theorem and boot-strap arguments, we know $u_s\in C^{2,\alpha}$ for any $0<\alpha<1.$ Moreover, we know $u_s>0$ by strong maximum principle.

We claim there is a constant $C>0$ such that $u_s<C$ for all $s\in (2,\frac{2n}{n-2}).$ Suppose this has been proved, then $u_s$ will be uniformaly bounded  in $C^{2,\alpha}. $ After taking a convergent subsequence of $u_s,$  the limit $u\in C^{2,\alpha}$ will satisfy  $$
[-4\frac{n-1}{n-2}\triangle u +(R_g-f(W_g)) u]=(\overline{\lim\limits_{s\rightarrow \frac{2n}{n-2}}}\mu_s) u^{\frac{n+2}{n-2}}.
$$
This implies $\overline{\lim\limits_{s\rightarrow \frac{2n}{n-2}}}\mu_s\geq\mathcal{Y}(M^n,\mathcal{C}),$ hence $\overline{\lim\limits_{s\rightarrow \frac{2n}{n-2}}}\mu_s=\mathcal{Y}(M^n,\mathcal{C}),$ and $u^{\frac{4}{n-2}}g$ is the solution of the problem.

To prove the claim, we argue by contradiction. Suppose there is a $s_i\rightarrow \frac{2n}{n-2}$ such that $m_i\triangleq\max u_{s_i}=u_{s_i}(x_i)\rightarrow +\infty,$ and $\{\mu_i\}$ is convergent. Take normal coordinates $y_j$ of $(M^n,g)$ around each $x_i,$ consider the new sequence of functions $v_i=m_i^{-1}u_{s_i}(m_i^{\frac{1-s_i}{2}}y)$ of $y.$ These functions $v_i$ have a subsequence (still denoted by $v_i$) converging to some $v\in C^{2,\alpha}(R^n)$ which is strictly positive and satisfies $v(0)=1$ and
\begin{equation}\label{yern}
-4\frac{n-1}{n-2}\triangle_{R^n} v=(\lim_{s_i\rightarrow \frac{2n}{n-2}}\mu_s) v^{\frac{n+2}{n-2}}.
\end{equation}
Moreover, it is not hard to show $\int_{R^n}|\nabla v|^2dy<\infty,$ and $\int_{R^n}|v|^{\frac{2n}{n-2}}dy<\infty.$ Then by multiplying both sides of the equation (\ref{yern}) by $v$ and a cut-off function, integrating by parts, we get
$$
\frac{\int_{R^n}|\nabla v|^2dy}{\int_{R^n}|v|^{\frac{2n}{n-2}}dy}=\lim\limits_{s_i\rightarrow \frac{2n}{n-2}}\mu_s.
$$
By the conformal invariance of the Yamabe invariant, we know $\lim\limits_{s_i\rightarrow \frac{2n}{n-2}}\mu_{s_i}\geq\mathcal{Y}(S^n).$ This is a contradiction with $\overline{\lim\limits_{s\rightarrow \frac{2n}{n-2}}}\mu_{s}\leq \mathcal{Y}(M^n,\mathcal{C})<\mathcal{Y}(S^n).$\\

Step 2. By the solution of Yamabe problem \cite{Sch}, we already know the equality
$\mathcal{Y}(M^n,\mathcal{C})=\mathcal{Y}(S^n)$   occurs if and only if
$(M^n,\mathcal{C})$ is conformal to the sphere $S^n$. So if $\mathcal{Y}_f(M^n,\mathcal{C})=\mathcal{Y}(S^n),$ the manifold is already conformal to the sphere. The problem is clearly solvable in this case.

Combining the both steps,  we know that the generalized Yamabe problem
is always solvable.
\end{proof}

  Particularly, one has the following result which was earlier obtained in \cite{GLe} (see Proposition 3 in \cite{GLe}):

 \begin{corollary}\label{2.7} If $\mathcal{Y}_f(M^n,\mathcal{C})>0,$ then there exists $\tilde{g}\in \mathcal{C}$ such that

 $$R_{\tilde{g}}-f(W_{\tilde{g}})>0.$$
\end{corollary}

To compute the invariants $\mathcal{Y}_f(M^n,\mathcal{C})$ on a manifold $M^n$, we need the following results.
\begin{lemma}
If
 $g\in \mathcal{C}$ achieves the minimum $\mathcal{Y}(M^n,\mathcal{C}),$ and $f(W_g)=const.,$
  then $g$ also achieves the minimum of $\mathcal{Y}_f(M^n,\mathcal{C}).$

\end{lemma}
\begin{proof}
Without loss of generality, we may normalize $g$ so that the volume $\int_{M^n}dv_{g}=1.$ Then $\mathcal{Y}(M^n,\mathcal{C})=R_g.$ For any $\tilde{g}=u^{\frac{4}{n-2}}g\in \mathcal{C}$ with unit volume $\int_{M^n}u^{\frac{2n}{n-2}}dv_g=1,$ we have

\begin{equation}\begin{split}\mathcal{F}_f(M^n,\tilde{g})&=\int_{M^n}4\frac{n-1}{n-2}|\nabla u|^2+(R_g-f(W_g))u^2dv_g\\
& \geq \mathcal{Y}(M^n,\mathcal{C})-f(W_g)\int_{M^n}u^2dv_g\\ & \geq \mathcal{Y}(M^n,\mathcal{C})-f(W_g)(\int_{M^n}u^{\frac{2n}{n-2}}dv_g)^{\frac{n-2}{n}}(\int_{M^n}dv_g)^{\frac{2}{n}}\\ &\geq R_g-f(W_g).
\end{split}
\end{equation}
Since we always have $$\mathcal{Y}_f(M^n,\mathcal{C})\leq R_g-f(W_g),$$ the proof is completed.

\end{proof}
It is known \cite{Leb} that the Riemannian metrics with negative constant scalar curvature or Einstein are Yamabe metrics. This implies

\begin{corollary}\label{coro2.5}If
 $g\in \mathcal{C}$ is of constant negative scalar curvature or an Einstein metric and $f(W_g)=const.,$
 then $g\in \mathcal{C}$
achieves $\mathcal{Y}_f(M^n,\mathcal{C})$

\end{corollary}

 The next lemma  handles the invariants $\mathcal{Y}_f$ under a connected sum operation.
  The result has already been know for the original Yamabe invariant (see \cite{K}). Due to the generality of $f$, the result of the current form can be applied in many circumstances.   We leave the proof to the appendix for completeness.  \begin{lemma} (See \cite{K})  \label{2.6}For $n\geq 3,$
  \begin{equation*}\mathcal{Y}_f(M_1 \# M_2)\geq
\left\{
\begin{split}
 \quad & -(|\mathcal{Y}_f(M_1)|^{\frac{n}{2}}+|\mathcal{Y}_f(M_2)|^{\frac{n}{2}})^{\frac{2}{n}}, \ \ \mbox{ if }  \mathcal{Y}_f(M_1)\leq 0\ \  \mbox{and} \ \  \mathcal{Y}_f(M_2)\leq 0,\\
  & \min\{\mathcal{Y}_f(M_1),\mathcal{Y}_f(M_2)\}, \ \ \ \  \ \ \ \ \ \  \mbox{ otherwise }.
  \end{split}
 \right.
\end{equation*}
\end{lemma}

\vskip 1cm

\section{Proof of Theorems}
\label{IN-0}

  Let us focus on dimension $4.$  Denote by $\lambda_{\max}(W_{+})$ (or $\lambda_{\max}(W_{-})$)
the largest eigenvalue of the Weyl operator $W_{+}$ (or $W_{-}$) acting on the
(anti-)self-dual 2-forms. Let the function $f$ in the previous section be given by $$f(W_g)=6\max\{\lambda_{\max}(W_{+}),\lambda_{\max}(W_{-})\}.$$
The  variant of the scalar curvature is:
\be \label{6}\sigma_g=R_g-6\max\{\lambda_{\max}(W_{+}),\lambda_{\max}(W_{-})\}.\ee
 Clearly, $\sigma_g$ is
not relevant to the orientation. 


Denote the new conformal invariant and differentiable invariant  by $\mathcal{GY}(M^4,\mathcal{C})$ and
$\mathcal{GY}(M^4):$ 
\be \begin{split}
&\mathcal{GY}(M^4,\mathcal{C})=\inf_{g\in \mathcal{C}}\frac{\int_{M^4}(R_g-6\max\{\lambda_{\max}(W_{+}),\lambda_{\max}(W_{-})\}) dv_{{g}}}{{(\int_{M^4}dv_{{g}})}^{\frac{1}{2}}}, \\
  & \mathcal{GY}(M^4)=\sup_{ \mathcal{C}}\mathcal{GY}(M^4, \mathcal{C}).
   \end{split}
\ee

\vskip 0.5cm
 Now we can prove Theorem \ref{t1.1}, which hinges on the authors' (with S.-H. Tang) recent classification theorem in \cite{CTZ}.

\begin{proof} of Theorem \ref{t1.1}.
By the definition of $\mathcal{GY}(M^4),$ there is a conformal class $\mathcal{C}$ such that $\mathcal{GY}(M^4, \mathcal{C})>0.$
Then from Corollary \ref{2.7}, there exists a $g\in \mathcal{C}$ such that $$\sigma_g=R_g-6\max\{\lambda_{\max}(W_{+}),\lambda_{\max}(W_{-})\}>0.$$
Since both $W_+$ and $W_-$ are of trace free,
this implies the sum of least two eigenvalues of $\frac{R_g}{12}+W_{\pm}$ is positive. In other words, $(M^4,g)$ has positive isotropic curvature.
The result then follows from the main theorem of \cite{CTZ}.
\end{proof}

Now we are in a position to prove Theorem \ref{t1.1'}.
\begin{proof} of Theorem \ref{t1.1'}. We shall show that the manifolds $M^4$ in Theorem \ref{t1.1'} have $\mathcal{GY}(M^4)>0.$
By Lemma \ref{l1.1}, there is a metric $\tilde{g}$ of unit volume in the conformal class $\mathcal{C}_g$ of $g$ achieving the value $\mathcal{GY}(M^4, \mathcal{C}_g)$ and
$$R_{\tilde{g}}-6\max\{\lambda_{\max}(W^{\tilde{g}}_{+}),\lambda_{\max}(W^{\tilde{g}}_{-})\}\equiv\mathcal{GY}(M^4, \mathcal{C}_g).$$
Hence, we have
\begin{equation}
\begin{split}
\mathcal{GY}(M^4, \mathcal{C}_g)&= \int_{M^4}R_{\tilde{g}}dv_{\tilde{g}}-6\int_{M^4} \max\{\lambda_{\max}(W^{\tilde{g}}_{+}),\lambda_{\max}(W^{\tilde{g}}_{-})\}dv_{\tilde{g}}\\
& \geq \mathcal{Y}(M^4,\mathcal{C}_g )-6 (\int_{M^4} [\max\{\lambda_{\max}(W^{\tilde{g}}_{+}),\lambda_{\max}(W^{\tilde{g}}_{-})\}]^2dv_{\tilde{g}})^{\frac{1}{2}}\\
& =\mathcal{Y}(M^4,\mathcal{C}_g )-6 (\int_{M^4} [\max\{\lambda_{\max}(W_{+}),\lambda_{\max}(W_{-})\}]^2dv_{{g}})^{\frac{1}{2}}\\
& >0,
\end{split}
\end{equation}
by the assumptions (i) and (ii) in  Theorem \ref{t1.1'}. Now the result follows from Theorem \ref{t1.1}.

\end{proof}

\begin{proof} of Theorem \ref{t1.1''}. By the assumptions and the above argument, we know $\mathcal{GY}(M^4, \mathcal{C}_g)=0,$ and the optimal metric $\tilde{g}\in \mathcal{C}_g$ satisfies $$R_{\tilde{g}}=6\max\{\lambda_{\max}(W^{\tilde{g}}_{+}),\lambda_{\max}(W^{\tilde{g}}_{-})\} \equiv const.>0.$$
This implies $(M^4,\tilde{g})$ has non-negative isotropic curvature.  Then we can appeal the results of Micallef-Wang \cite{MW}. The reason is the following.

Denote $P=\frac{R}{6}I-W$ and $P_{\pm}=\frac{R}{6}I_{\Lambda^2_{\pm}}-W_{\pm}.$ Then the positive (nonnegative) isotropic curvature condition is equivalent to $P>0(\geq 0).$ We run the Ricci flow on $M^4$ with $\tilde{g}$ as initial data, and get a solution $\tilde{g}(t)$. It is known from \cite{H} that non-negative isotropic curvature is preserved. Combining the main theorem in \cite{CTZ} and the assumption, we know that the manifold can not admit metrics of positive isotropic curvature. By a strong maximum principle argument(Theorem 4.6 in \cite{MW}), the kernels $Ker P_{\pm}$ is invariant under parallel translation, and invariant in time. By changing orientation, we may assume $dim Ker P_{+}\geq 1.$ This will imply that  the universal cover $\tilde{M}^4$ is a K$\ddot{a}$hler manifold with positive scalar curvature (see Theorem 4.9 (c) in \cite{MW}).

Applying de Rham decomposition theorem for K$\ddot{a}$hler manifolds, we know that one of the following occurs:\\
(1) If $(M^4,\tilde{g})$ is locally irreducible, $(M^4,\tilde{g})$ is biholomorphic to $CP^2$ and $\tilde{g}_{t}$ is a K$\ddot{a}$hler metric with positive Chern class (see Theorem 4.10 (d) in \cite{MW}). Since the scalar curvature $R_{\tilde{g}}=const.>0,$ and $b_2(M^4)=1$ in this case, we know $\tilde{g}$ is K$\ddot{a}$hler-Einstein, and hence a homothetic of the standard Fubini-Study metric. This is case a) of Theorem \ref{t1.1''} \\
(2) If $(M^4,\tilde{g})$ is locally reducible, the universal cover of $(M^4,\tilde{g})$ is isometric to $(M_1,g_1)\times (M_2,g_2),$ where $(M_i,g_i)$ is a two-dimensional manifold. Since the scalar curvature $R_{\tilde{g}}=const.>0,$ the Gaussian curvature $k_i$ of $g_i$ must be a constant and satisfies $k_1+k_2>0.$ 
This is case b) of Theorem \ref{t1.1''}.

\end{proof}

Finally, we  compute the generalized Yamabe invariant $\mathcal{GY}(M^4)$ for several concrete 4-manifolds. (For Yamabe invariant, the readers may refer
to the survey \cite{Le}.)

\vskip 0.5cm
 i)  $\mathcal{GY}(S^3\times S^1)=\mathcal{Y}(S^4).$

 Indeed,  there is a sequence of conformally flat structures
 $\mathcal{C}_i$  such that $\mathcal{GY}(\mathcal{C}_i)(=\mathcal{Y}(\mathcal{C}_i))\rightarrow
 \mathcal{Y}(S^4).$
\vskip 0.3cm
 ii) $\mathcal{GY}(CP^2)=0.$

 Let $\mathcal{C}_{FS}$ be the conformal class of the Fubini-Study metric, Corollary \ref{coro2.5} implies
 $$
 \mathcal{GY}(CP^2, \mathcal{C}_{FS})=\mathcal{F}_f(CP^2,g_{FS})=0.
 $$
 Then we have $\mathcal{GY}(CP^2)\geq 0.$ Combining this with Theorem \ref{t1.1}, we have $\mathcal{GY}(CP^2)=0.$
\vskip 0.3cm
  iii) $\mathcal{GY}(T^4)=0.$
  The same argument as ii).
\vskip 0.3cm
  iv) Let $M$ be
  a compact complex hyperbolic manifold of real dimension 4, equipped with the Bergman metric $g_{ch}.$
  Then $\mathcal{GY}(M)=\mathcal{Y}(M)=\mathcal{F}(M,g_{ch}).$
  Indeed, C. Lebrun \cite{Leb} has shown that the negative K$\ddot{a}$hler-Einstein metrics (on real dimension 4) achieve the Yamabe invariant!
  The metrics achieving the Yamabe invariant are called \emph{supreme Einstein} in his paper.
  Unfortunately, we still do not know whether the 4-dimensional real hyperbolic metric is \emph{supreme Einstein}, whereas it is the case in dimension 3.
\vskip 0.3cm
v)  $\mathcal{GY}(k_1CP^2 \#  k_2\bar{CP}^2 \#  k_3 T^4\# k_4 S^3\times S^1)=0,$ where integers $k_i\geq 0,$ $k_1+k_2+k_3\geq 1.$

\vskip 0.5cm
\begin{remark} There are many other possible generalizations of the Yamabe invariant for oriented 4-manifolds. For instance, associated to every modified scalar curvature (introduced  in \cite{GLe}) of the form $R_g-f$ with a function $f$ of conformal weight $-2$, one has a modified Yamabe invariant. Let us consider another special case that  $\sigma_g^+=R_g-6\lambda_{\max}(W_{+})$. The corresponding invariant is denoted by
$$\mathcal{GY}_{+}(M,\mathcal{C})=\inf_{\tilde{g}\in \mathcal{C}}\frac{\int_{M}\sigma_{\tilde{g}}^+}{V_{\tilde{g}}^{\frac{1}{2}}}, \ \ \ \ \ \mathcal{GY}_{+}(M)=\sup_{\mathcal{C}}\mathcal{GY}_{+}(M,\mathcal{C}).$$
As before, $\mathcal{GY}_{+}(M)>0$ if and only if $M$ admits a metric such that $P_{+}=\frac{R}{6}I_{\Lambda^2_{+}}-W_{+}>0$ in the notation of Micallef and Wang \cite{MW}. A consequence of  Bochner formula argument of Theorem 2.1(a) in Micallef and Wang \cite{MW} is that  $\mathcal{GY}_{+}(M)>0$ implies $b_{2}^{+}(M)=0.$

We remark that  the same conclusion ($b_{2}^{+}=0$) was also obtained by M. Gursky (\cite{Gur} Theorem 3.3), where $R_g-6\lambda_{\max}(W_{+})$ is replaced by $R_g-2\sqrt{6}|W_{+}|.$

Except for the manifolds with positive isotropic curvature, $\bar{CP}^2$ is a good example with $\mathcal{GY}_+>0.$
  By Lemma \ref{2.6}, one can show $$\mathcal{GY}_{+}(l_1 \bar{CP}^2\# l_2 S^3\times S^1)>0,$$ and $$\mathcal{GY}_{+}(l_1 \bar{CP}^2\# l_2 CP^2 \# l_3 T^4 \# l_4 K3\# l_5 S^3\times S^1 )=0,$$ where integers $l_i\geq 0$ and $l_2+l_3+l_4>0.$

  Of course, $\mathcal{GY}_{-}$ can also be defined. Actually reversing the orientation interchanges $\mathcal{GY}_{-}$ and $\mathcal{GY}_{+}.$

 In view of Theorem \ref{t1.1}, a natural question may be raised that whether  an oriented 4-manifold with positive $\mathcal{GY}_+$ is  a connected sum of several  $\bar{CP}^2$s and a manifold with positive $\mathcal{GY}$ (appeared  in our Theorem \ref{t1.1}).

\end{remark}
\section{Appendix}
The proof of Lemma \ref{2.6} is analogous to the Yamabe invariant case. Our argument follows the line of \cite{K} with slight modifications.
\begin{proof} of lemma \ref{2.6}. Note that the right hand side of the desired inequality is precisely $\mathcal{Y}_{f}(M_1\sqcup M_2),$ where $M_1\sqcup M_2$ is the disjoint union of $M_1$ and $M_2.$

For any fixed small $\varepsilon>0,$ take a conformal class $\tilde{\mathcal{C}}$ on $M_1\sqcup M_2,$ such that $$\mathcal{Y}_f(M_1\sqcup M_2,\tilde{\mathcal{C}})>\mathcal{Y}_f(M_1\sqcup M_2)-\frac{\varepsilon}{4}.$$
Pick $P_1\in M_1, P_2\in M_2,$
and perturb $\tilde{\mathcal{C}}$ a little so that the obtained conformal class $\mathcal{C}$ is conformal flat near $P_1$ and $P_2,$  and
\be\label{pur}
\mathcal{Y}_f(M_1\sqcup M_2,{\mathcal{C}})>\mathcal{Y}_f(M_1\sqcup M_2)-\frac{\varepsilon}{2}.
\ee
The argument is the following. Let $\tilde{g}\in \tilde{\mathcal{C}}$ be a fixed metric, take normal coordinate system of radius
$\delta$ near $P_1$(and $P_2$) so that $\tilde{g}_{ij}=\delta_{ij}+\eta_{ij}$ where $\eta_{ij}=O(r^2).$
 Let $\xi:\mathbb{R}\rightarrow\mathbb{R}$ be a function satisfying $0\leq \xi\leq 1,$ $\xi\mid_{(-\infty,\frac{1}{2}]}=0,$
  $\xi\mid_{[1,\infty)}=1.$ We construct a new metric $g=\tilde{g}-(1-\xi(\frac{r}{\delta}))\eta.$ Then $g$ is flat on  balls of radius $\frac{\delta}{2}$ centered at $P_1$ and $P_2,$ and $g$ coincides $\tilde{g}$ outside $B(P_1,{\delta}{})\cup B(P_2,{\delta})$ and $$
  |g-\tilde{g}|_{\tilde{g}}<C\delta^2, \ \ \ |Rm(g)|_{B(P_1,{\delta}{})\cup B(P_2,{\delta})}<C,
  $$
  where $C$ is independent of $\delta.$ It is now clear that for any function $u,$ we have
  $$
  (1-C\delta^2)\frac{\int |\nabla u|^2_{\tilde{g}} dv_{\tilde{g}}}{(\int u^{\frac{2n}{n-2}} dv_{\tilde{g}})^{\frac{n-2}{n}}}\leq \frac{\int |\nabla u|^2_g dv_g}{(\int u^{\frac{2n}{n-2}} dv_g)^{\frac{n-2}{n}}}\leq(1+C\delta^2)\frac{\int |\nabla u|^2_{\tilde{g}} dv_{\tilde{g}}}{(\int u^{\frac{2n}{n-2}} dv_{\tilde{g}})^{\frac{n-2}{n}}},
  $$
  and \be \label{2.9} \mid\int(R_g-f(W_g))u^2dv_g-\int(R_{\tilde{g}}-f(W_{\tilde{g}}))u^2dv_{\tilde{g}}\mid\leq C(\int u^{\frac{2n}{n-2}}dv_{\tilde{g}})^{\frac{n-2}{n}}\delta^{2},\ee

  \be \label{2.10} \mid(\int u^{\frac{2n}{n-2}} dv_{\tilde{g}})^{\frac{n-2}{n}}-(\int u^{\frac{2n}{n-2}} dv_{{g}})^{\frac{n-2}{n}}\mid\leq C(\int u^{\frac{2n}{n-2}}dv_{\tilde{g}})^{\frac{n-2}{n}}\delta^{2}.\ee
    This implies

    \be \label{2.11}\begin{split}
   &\mid\frac{\int(R_g-f(W_g))u^2+4\frac{n-1}{n-2}|\nabla u|^2dv_g}{(\int u^{\frac{2n}{n-2}}dv_g)^{\frac{n-2}{n}}}-
   \frac{\int(R_{\tilde{g}}-f(W_{\tilde{g}}))u^2+4\frac{n-1}{n-2}|\nabla u|^2dv_{\tilde{g}}}{(\int u^{\frac{2n}{n-2}}dv_{\tilde{g}})^{\frac{n-2}{n}}}\mid\\
   & \leq C\delta^2\frac{\int|\nabla u|^2dv_{\tilde{g}}}{(\int u^{\frac{2n}{n-2}}dv_{\tilde{g}})^{\frac{n-2}{n}}}+C\delta^2.
   \end{split}\ee

    Hence we get $$\mathcal{Y}_f(\mathcal{C})\leq \mathcal{Y}_f(\tilde{\mathcal{C}})+C\delta^2$$ immediately.
    Since the curvature of $g$ is uniformally bounded (independent of $\delta$), a minimizing sequence $u_i$ for $\mathcal{Y}_f(\mathcal{C})$ must have
    $$\frac{\int |\nabla u_i|^2_{\tilde{g}} dv_{\tilde{g}}}{(\int u^{\frac{2n}{n-2}}dv_{\tilde{g}})^{\frac{n-2}{n}}}\leq C,$$
     therefore
     $$ \mathcal{Y}_f(\tilde{\mathcal{C}})-C\delta^2\leq \mathcal{Y}_f(\mathcal{C})\leq \mathcal{Y}_f(\tilde{\mathcal{C}})+C\delta^2.
    $$
    In particular, choosing small $\delta,$ $(\ref{pur})$ follows.   Now conformally scaling the metric $g,$ so that the obtained metric $\bar{g}=e^{-\xi(\frac{2r}{\delta}) \log r^2}$ turns  neighborhoods of $P_1$ and $P_2$ into two infinite half cylinder $S^{n-1}\times [0,\infty).$
    Denote the complement of the two half cylinders by $(M^0,g_0),$ here $M_0=M^0_1\sqcup M^0_2$ has two connected components .
    For large $l>0,$ truncating the two $S^{n-1}\times [l,\infty)$ and gluing along $ S^{n-1}\times [0,l],$ we get a Riemannian manifold $({M}_l,{g}_l),$
    where $M_l=M^0_1\cup S^{n-1}\times [0,l]\cup M^0_2.$ Let $\mathcal{C}_l=[g_l]$ be the conformal class of $g_l.$
    By the definition of
$\mathcal{Y}_f({\mathcal{C}}),$  there is a smooth positive function $u_l$ on $M_l$ such that
    \be
\label{dadao}\int (R_g-f(W_g))u^2+4\frac{n-1}{n-2}|\nabla u|^2
dv_{g_l}<\mathcal{Y}_f(M_l,C_l)+\frac{\varepsilon}{2},\ \ \ \ \ \
\int u_l^{\frac{2n}{n-2}}dv_{g_l}=1. \ee
    By mean value theorem, it is clear that there is  a $t_l\in [0,l]$ such that
\be \label{et2}
    \int_{S^{n-1}\times\{t_l\}}|du_l|^2+u_l^2dv_{S^{n-1}} \leq \frac{C}{l}
    \ee
    where $C$ is independent of $l.$
     Cut off $M_l$ along the section $S^{n-1}\times \{t_l\},$ and attach two half infinite cylinders to it again, we get $ M_1\backslash\{P_1\} \sqcup M_2\backslash\{P_2\} .$ Extending the function $u_l$ over to the two half cylinders, the resulting function is denoted by $U_l$ is given by
    \begin{equation*}U_l(x,t)=
\left\{
\begin{split}
  & (1-t)u_l\mid_{S^{n-1}\times\{t_l\}}(x,t_l)\ \ \ \ \mbox{for} \ \ \ (x,t)\in S^{n-1}\times [0,1]\\
  & 0 \ \ \ \ \ \ \ \ \ \ \ \ \ \ \ \ \ \ \ \ \ \ \ \ \mbox{for} \ \ \ \ \ (x,t)\in S^{n-1}\times [1,\infty) .
  \end{split}
 \right.
\end{equation*}
From (\ref{dadao}), (\ref{et2}) and (\ref{pur}), we have
\be
    \label{dadao1}\int  (R_g-f(W_g))U_l^2+4\frac{n-1}{n-2}|\nabla U_l|^2
dv_{g}<\mathcal{Y}_f(M_l,\mathcal{C}_l)+\frac{\varepsilon}{2}+\frac{C}{l},\ \ \ \ \ \
\int_{}U_l^{\frac{2n}{n-2}}dv_{g}>1,\ee
hence
\be (\mathcal{Y}_f(M_1\sqcup M_2)-\frac{\varepsilon}{2})(\int U_l^{\frac{2n}{n-2}}dv_{g})^{\frac{n-2}{n}} \leq {\mathcal{Y}_f(M_{l},\mathcal{C}_l)}+\frac{C}{l}+\frac{\varepsilon}{2}.\ee
 By Sobolev imbedding theorem, we have
  $$ \parallel u_l\parallel_{L^{\frac{2(n-1)}{n-1-2}}(S^{n-1}\times \{t_l\})} \leq C\parallel u_l\parallel_{W^{1,2}(S^{n-1}\times \{t_l\})}\leq \frac{C}{\sqrt{l}},$$
  and $$1<\int U_l^{\frac{2n}{n-2}}dv_{g} \leq 1+\frac{C}{l^{\frac{n}{n-2}}}.$$
  Consequently, we have
\be \mathcal{Y}_f(M_1\sqcup M_2)\leq {\mathcal{Y}_f(M_{l},\mathcal{C}_l)}+\frac{C}{l}+2\varepsilon\leq {\mathcal{Y}_f(M_1\#M_2)}+\frac{C}{l}+2\varepsilon.\ee

 The result follows by taking $l\rightarrow \infty$ and  $\varepsilon \rightarrow 0.$\end{proof}


\end{document}